\documentclass[10pt]{amsart}
\usepackage{graphicx,amscd,color,amsmath,amsfonts,amssymb,geometry,soul,hyperref}

\newtheorem{theorem}{Theorem}
\newtheorem*{theorem*}{Theorem}
\theoremstyle{plain}

\newtheorem{lemma}{Lemma}

\newtheorem{remark}{Remark}

\numberwithin{equation}{section}
\newtheorem*{KSZ}{Generalized Kahane--Salem--Zygmund}
\begin{document}
\title[Some applications of the H\"{o}lder inequality for mixed sums]{Some
applications of the H\"{o}lder inequality for mixed sums}
\author[Albuquerque]{N. Albuquerque}
\address[N. Albuquerque]{Departamento de Matem\'{a}tica, Universidade
Federal da Para\'{\i}ba, 58.051-900 - Jo\~{a}o Pessoa, Brazil.}
\email{ngalbqrq@gmail.com}
\author[Nogueira]{T. Nogueira}
\address[T. Nogueira]{Departamento de Matem\'{a}tica, Universidade Federal
da Para\'{\i}ba, 58.051-900 - Jo\~{a}o Pessoa, Brazil.}
\email{tonykleverson@gmail.com}
\author[N\'u\~nez]{D. N\'u\~nez-Alarc\'on}
\address[D. N\'u\~nez]{Departamento de Matem\'{a}tica, Universidade Federal
de Pernambuco, 50.740-560 - Recife, Brazil.}
\email{danielnunezal@gmail.com}
\author[Pellegrino]{D. Pellegrino}
\address[D. Pellegrino]{Departamento de Matem\'{a}tica, Universidade Federal
da Para\'{\i}ba, 58.051-900 - Jo\~{a}o Pessoa, Brazil.}
\email{dmpellegrino@gmail.com and pellegrino@pq.cnpq.br}
\author[Rueda]{P. Rueda}
\address[P. Rueda]{Departamento de An\'{a}lisis Matem\'{a}tico, Universidad
de Valencia, 46100 Burjassot, Valencia.}
\email{pilar.rueda@uv.es}
\thanks{D. Pellegrino is supported by CNPq}

\begin{abstract}
We use the H\"{o}lder inequality for mixed exponents to prove some optimal
variants of the generalized Hardy--Littlewood inequality for $m$-linear
forms on $\ell _{p}$ spaces with mixed exponents. Our results extend recent
results of Araujo \textit{et al.}
\end{abstract}

\maketitle

\section{Introduction}

\bigskip Several extensions and generalizations of the famous H\"{o}lder's
inequality have appeared along the time. One of these extensions is the H{%
\"{o}}lder inequality for mixed $L_{p}$ spaces. This inequality seems to
have been proved for the first time in 1961, by A. Benedek and R. Panzone
\cite{bene}, although its roots seem to go back to the work of W.A.J.
Luxemburg \cite{lux}. Very recently, this inequality was rediscovered in
Functional Analysis in a somewhat simplified form of an interpolative result
(see \cite{alb, isr}). In this note we obtain an application of this H\"{o}%
lder inequality to the Hardy--Littlewood inequality for multilinear forms.

From now on, $\mathbb{K}$ $=\mathbb{R}$ or $\mathbb{C}$ and $\mathbf{i}%
=(i_1,\dots, i_m) \in \mathbb{N}^{m}$ is a multi-index. For scalar matrices $%
a^{(k)}= \left( a_{\mathbf{i}}^{(k)}\right) _{i_1,\dots,i_m=1}^{n}$, $%
k=1,\dots,N$, we consider the coordinate product
\begin{equation*}
a^{(1)}\cdot \ldots \cdot a^{(N)} := \left( a_{\mathbf{i}}^{(1)}\cdots a_{%
\mathbf{i}}^{(N)}\right)_{i_1,\dots,i_m=1}^{n}.
\end{equation*}

The H\"{o}lder inequality for mixed sums is stated as follows (see also \cite%
[Theorem 2.49]{four}):

\begin{theorem}[H\"{o}lder's inequality for mixed $\ell _{p}$ spaces]
\label{GenHol}(\cite{lux, bene}) Let $r_{j},q_{j}(k) \in (0, + \infty ]$,
for $j=1,\dots,m,\ k=1,\dots,N$, be such that
\begin{equation*}
\frac{1}{r_{j}}=\frac{1}{q_{j}(1)}+\cdots +\frac{1}{q_{j}(N)},\quad j\in
\{1,2,\ldots ,m\}
\end{equation*}%
and let $a^{(k)}=\left( a_{\mathbf{i}}^{(k)}\right) _{i_1,\dots,i_m=1}^{n}$
be a scalar matrix for all $k=1,\dots,N$. Then
\begin{align*}
& \left( \sum_{i_{1}=1}^{n}\left( \dots \left( \sum_{i_{m}=1}^{n}|a_{\mathbf{%
i}}^{(1)}\cdot \ldots \cdot a_{\mathbf{i}}^{(N)}|^{r_{m}}\right) ^{\frac{%
r_{m-1}}{r_{m}}}\dots \right) ^{\frac{r_{1}}{r_{2}}}\right) ^{\frac{1}{r_{1}}%
} \\
& \leq \prod_{k=1}^{N}\left[ \left( \sum_{i_{1}=1}^{n}\left( \dots \left(
\sum_{i_{m}=1}^{n}|a_{\mathbf{i} }^{(k)}|^{q_{m}(k)}\right) ^{\frac{%
q_{m-1}(k)}{q_{m}(k)}}\dots \right) ^{\frac{q_{1}(k)}{q_{2}(k)}}\right) ^{%
\frac{1}{q_{1}(k)}}\right].
\end{align*}
\end{theorem}

\bigskip As usual, for a positive integer $N$ we define $\ell _{\infty
}^{N}= \mathbb{K}^{N}$ endowed with the supremum norm ; by $e_{j}$ we denote
the canonical vectors which entries are $1$ at $j$-th position and 0
otherwise. We define $X_{p}:=\ell _{p}$, for $1\leq p<\infty $, and $%
X_{\infty }:=c_{0}$. For $\mathbf{p}:= (p_1, \ldots , p_m) \in [1, +
\infty]^m$ let
\begin{equation*}
\left| \frac{1}{\mathbf{p}}\right| := \frac{1}{p_1} + \cdots + \frac{1}{p_m}.
\end{equation*}

\bigskip

We present a brief chronology of well known results concerning the
``Hardy-Littlewood'' inequality.

\begin{theorem*}[Bohnenblust--Hille inequality (\protect\cite{bh}, 1931)]
There exists a (optimal) constant $C_{m,\infty }^{\mathbb{K}}\geq 1$ such
that, for every continuous $m$--linear form $T:c_{0}\times \cdots \times
c_{0}\rightarrow \mathbb{K}$,
\begin{equation}
\left( \sum_{i_{1},\ldots ,i_{m}=1}^{\infty }\left\vert T(e_{i_{1}},\ldots
,e_{i_{m}})\right\vert ^{\frac{2m}{m+1}}\right) ^{\frac{m+1}{2m}}\leq
C_{m,\infty }^{\mathbb{K}}\left\Vert T\right\Vert .  \label{u88}
\end{equation}
Moreover, the exponent $\frac{2m}{m+1}$ is optimal.
\end{theorem*}

\begin{theorem*}[Hardy--Littlewood \protect\cite{hl} and Praciano-Pereira
\protect\cite{pp} (1934 and 1981)]
Let $\left\vert \frac{1}{\mathbf{p}}\right\vert \leq \frac{1}{2}$. There
exists a (optimal) constant $C_{m,\mathbf{p}}^{\mathbb{K}}\geq 1$ such that,
for every continuous $m$-linear form $T:X_{p_{1}}\times \cdots \times
X_{p_{m}} \to \mathbb{K}$,
\begin{equation}
\left( \sum_{i_{1},\ldots ,i_{m}=1}^{\infty }\left\vert T(e_{i_{1}},\ldots
,e_{i_{m}})\right\vert ^{\frac{2m}{m+1-2\left\vert \frac{1}{\mathbf{p}}%
\right\vert }}\right) ^{\frac{m+1-2\left\vert \frac{1}{\mathbf{p}}%
\right\vert }{2m}} \leq C_{m,\mathbf{p}}^{\mathbb{K}}\left\Vert T\right\Vert
.  \label{i99}
\end{equation}
Moreover, the exponent $\frac{2m}{m+1-2\left| \frac{1}{\mathbf{p}} \right|}$
is optimal.
\end{theorem*}

\begin{theorem*}[Hardy--Littlewood \protect\cite{hl} and
Dimant--Sevilla-Peris \protect\cite{dim} (1934 and 2013)]
Let $\frac{1}{2}\leq \left\vert \frac{1}{\mathbf{p}} \right\vert <1$. There
exists a (optimal) constant $D_{m,\mathbf{p}}^{\mathbb{K}}\geq 1$ such that
\begin{equation}
\left( \sum_{i_{1},\dots ,i_{m}=1}^{\infty }\left\vert T(e_{i_{1}},\dots
,e_{i_{m}})\right\vert ^{\frac{1}{1-\left\vert \frac{1}{\mathbf{p}}
\right\vert }}\right) ^{1-\left\vert \frac{1}{\mathbf{p}}\right\vert} \leq
D_{m,\mathbf{p}}^{\mathbb{K}}\Vert T\Vert  \label{i99jagsytsb}
\end{equation}
for every continuous $m$-linear form $T:X_{p_{1}}\times \cdots \times
X_{p_{m}}\rightarrow \mathbb{K}$. Moreover, the exponent $\left(1-\left\vert
\frac{1}{\mathbf{p}}\right\vert\right)^{-1}$ is optimal.
\end{theorem*}

The Hardy Littlewood inequality for $m$-linear forms on $\ell _{p}$ spaces
with mixed exponents and $\left| \frac{1}{\mathbf{p}}\right| \leq \frac{1}{2}
$ was proved in \cite[Theorem 1.2]{alb} (see also \cite[p. 3729]{alb} for
more details) and can be stated as follows:

\begin{theorem}[Generalized Hardy--Littlewood inequality]
\label{main111}(\cite{alb}) Let $m\geq 2$ be a positive integer, $\mathbf{p}%
:= (p_1, \ldots , p_m) \in [1, + \infty]^m$ such that $\left| \frac{1}{%
\mathbf{p}}\right| \leq \frac{1}{2}$ and $\mathbf{s} :=
(s_{1},\dots,s_{m})\in \left[ \left( 1-\left| \frac{1}{\mathbf{p}} \right|
\right)^{-1},2\right]^{m}.$ The following assertions are equivalent:

\noindent(a) There exists $D_{m,\mathbf{s},\mathbf{p}}^{\mathbb{K}}\geq 1$
satisfying,
\begin{equation}
\left( \sum_{i_{1}=1}^{n}\left( \ldots \left( \sum_{i_{m}=1}^{n}\left\vert
A\left( e_{i_{1}},\ldots ,e_{i_{m}}\right) \right\vert ^{s_{m}}\right) ^{
\frac{s_{m-1}}{s_{m}}}\ldots \right) ^{\frac{s_{1}}{s_{2}}}\right) ^{\frac{1
}{s_{1}}}\leq D_{m,\mathbf{s},\mathbf{p}}^{\mathbb{K}}\Vert A\Vert .
\label{q3}
\end{equation}
for every continuous $m$-linear map $A:\ell _{p_1}^{n}\times \cdots \times
\ell _{p_m}^{n}\rightarrow \mathbb{K}$ and all positive integers $n$.

\noindent(b) $\displaystyle
\frac{1}{s_{1}}+\cdots +\frac{1}{s_{m}}\leq \frac{m+1}{2}-\left| \frac{1}{%
\mathbf{p}}\right|. $
\end{theorem}

\bigskip

In this paper we are mainly interested in what happens with (\ref{q3}) (and
consequently with \eqref{u88}, \eqref{i99} and \eqref{i99jagsytsb}) when
\begin{equation*}
\frac{1}{s_{1}}+\cdots +\frac{1}{s_{m}}>\frac{m+1}{2}-\left\vert \frac{1}{
\mathbf{p}}\right\vert .
\end{equation*}
From the statement of Theorem \ref{main111} it is obvious that the constants
involved will gain a dependence on $n$ when smaller exponents are
considered. We are mainly concerned in calculating the exact dependence on $%
n $ depending on the new exponents.

In \cite{archiv} a similar problem is considered having the
\textquotedblleft classical\textquotedblright\ Hardy--Littlewood inequality
(when all exponents are the same). The main result of \cite{archiv} is the
following:

\begin{theorem}
\label{9900}(\cite{archiv}) Let $m\geq 2$ be a positive integer.

(a) If $\left( r,p\right) \in \left( \lbrack 1,2]\times \lbrack 2,2m)\right)
\cup \left( \lbrack 1,+ \infty )\times \lbrack 2m,+ \infty ]\right) $, then
there is a constant $D_{m,r,p}^{\mathbb{K}}\geq 1$ such that
\begin{equation*}
\left( \sum_{i_{1},\dots ,i_{m}=1}^{n}\left\vert T(e_{i_{1}},\dots
,e_{i_{m}})\right\vert ^{r}\right) ^{\frac{1}{r}}\leq D_{m,r,p}^{\mathbb{K}%
}n^{\max \{\frac{2mr+2mp-mpr-pr}{2pr},0\}}\Vert T\Vert
\end{equation*}%
for all $m$-linear forms $T:\ell _{p}^{n}\times \cdots \times \ell
_{p}^{n}\rightarrow \mathbb{K}$ and all positive integers $n$. Moreover, the
exponent $\max \{\left( 2mr+2mp-mpr-pr\right) /2pr,0\}$ is optimal.

(b) If $(r,p)\in \lbrack 2,+ \infty )\times (m,2m]$, then there is a
constant $D_{m,r,p}^{\mathbb{K}}\geq 1$ such that
\begin{equation*}
\left( \sum_{i_{1},\dots ,i_{m}=1}^{n}\left\vert T(e_{i_{1}},\dots
,e_{i_{m}})\right\vert ^{r}\right) ^{\frac{1}{r}}\leq D_{m,r,p}^{\mathbb{K}%
}n^{\max \{\frac{p+mr-rp}{pr},0\}}\Vert T\Vert
\end{equation*}%
for all $m$-linear forms $T:\ell _{p}^{n}\times \cdots \times \ell
_{p}^{n}\rightarrow \mathbb{K}$ and all positive integers $n$. Moreover, the
exponent $\max \{\left( p+mr-rp\right) /pr,0\}$ is optimal.
\end{theorem}

\begin{remark}
\label{7654} Theorem \ref{9900} recovers the famous Bohnenblust--Hille
inequality (see \cite{bh}) when $r=\frac{2m}{m+1}$ and $p=+ \infty $. When $%
r=\frac{2mp}{mp+p-2m}$ and $p\geq 2m$ it is recovered the Hardy--Littlewood
/ Praciano-Pereira inequality (see \cite{hl, pp}). For $r=\frac{p}{p-m}$ and
$m<p<2m$ we get the Hardy--Littlewood / Dimant--Sevilla-Peris inequality
(see \cite{hl,dim}).
\end{remark}


\begin{remark}
\label{APps} It is worth noting that item $(a)$ of the previous result holds
for a more general situation: since H\"{o}lder's inequality is still valid
for exponents between $0$ and $1$, we may have $r \in (0,1)$; also,
following the lines of \cite{archiv}, we may consider a version with
different values of $p$. More precisely, theorem \ref{9900} item (a) may be
read as follows: let $m\geq 2$ be a positive integer and $\mathbf{p}= (p_1,
\ldots , p_m)$. If $(r,\mathbf{p}) \in(0,2]\times [2,2m)^m \cup (0,+\infty)
\times [2m,+ \infty]^m$, then there is a constant $D_{m,r,\mathbf{p}}^{%
\mathbb{K}} \geq 1$ (not depending on $n$) such that
\begin{equation*}
\left( \sum_{j_{1},\ldots ,j_{m}=1}^{n}\left\vert T(e_{j_{1}},\ldots
,e_{j_{m}})\right\vert ^{r}\right) ^{\frac{1}{r}} \leq D_{m,r,\mathbf{p}}^{%
\mathbb{K}} \cdot n^{ \max \left\{\frac{m}{r} - \frac{m+1}{2} + \left| \frac{%
1}{\mathbf{p}} \right|,0\right\} }\Vert T\Vert,
\end{equation*}
for all $m$-linear forms $T:\ell_{p_{1}}^{n}\times \cdots \times
\ell_{p_{m}}^{n} \to \mathbb{K}$ and all positive integers $n$. Moreover,
the exponent $\max\left\{\frac{m}{r} - \frac{m+1}{2} + \left| \frac{1}{%
\mathbf{p}}\right|,0 \right\}$ is sharp. In this situation, as expected, the
exponent goes to $+\infty $ when $r$ goes to $0$.
\end{remark}

\bigskip

For the sake of clarity we fix the following notation: for $%
\rho,r_1,\dots,r_m \in (0,+ \infty)$, let us define $M_{<}^{\rho}
:=\{j\in\{1,\dots,m\}\,:\, r_j < \rho\} , \, M_{\geq}^{\rho} :=
\{1,\dots,m\} \setminus M_{<}^{\rho}$. We shall denote by $\left|
M_{<}^{\rho} \right|$ the cardinality of $M_{<}^{\rho}$. A particular case
of $\rho$ will be useful: $M_<^{\text{HL}}:= M_<^{2m/(m+1-2\left|\frac 1{%
\mathbf{p}}\right|) }$.

\begin{theorem}
\label{thm_uni} Let $m\geq 2$ be an integer, $\mathbf{p}:=(p_1, \ldots, p_m)
\in [2, + \infty]^{m}$, and $\mathbf{r} := \left( r_{1},\dots,r_{m} \right)
\in (0,+ \infty)^m$. There exists a constant $D_{m,\mathbf{p}}^{\mathbb{K}%
}\geq 1$ such that for all positive integer $n$ and for all bounded $m$%
-linear forms $T:\ell _{p_1}^{n}\times \cdots \times \ell_{p_m}^{n} \to
\mathbb{K}$ the following inequality holds
\begin{equation}
\left( \sum_{i_{1}=1}^{n} \left( \dots \left( \sum_{i_{m}=1}^{n} \left\vert
T(e_{i_{1}},\dots ,e_{i_{m}})\right\vert^{r_{m}} \right) ^{\frac{r_{m-1}}{
r_{m}}}\dots \right) ^{\frac{r_{1}}{r_{2}}}\right) ^{\frac{1}{r_{1}}} \leq
D_{m, \mathbf{p}}^{\mathbb{K}} n^s \Vert T\Vert ,  \label{thm_uni_ineq}
\end{equation}
where the exponent $s$ is given by:

\begin{enumerate}
\item If $\mathbf{p} \in [2, 2m]^m $ then, $s=\displaystyle \sum_{j \in
M_{<}^{2}} \frac{1}{ r_j}+ \left| \frac{1}{\mathbf{p}} \right| - \frac{1}{2}%
\left(|M_<^{2}|+1 \right)$. When $M_{<}^{2}= \{1,\dots,m\}$ the exponent $s$
is optimal.

\item If $\left|\frac{1}{\mathbf{p}}\right|\leq \frac 12$ (in particular, if
$\mathbf{p} \in [2m, +\infty]^m$) then, $s=\displaystyle\sum_{j \in M_{<}^{%
\text{HL}}} \frac{1}{r_j} - \frac{m+1-2\left| \frac{1}{\mathbf{p}} \right|}{%
2m} \cdot |M_<^{\text{HL}}|$. When $M_{<}^{\text{HL}} = \{1,\dots,m\}$ or $%
M_{<}^{\text{HL}} = \emptyset$ the exponent $s$ is optimal.
\end{enumerate}
\end{theorem}

Note that when $p_1=\cdots=p_m=2m$ in (1) and (2), the sets $M_<^{HL}=M_<^2$
coincide and so do both exponents $s$.

The paper is organized as follows. In Section \ref{sec_proof} we prove
Theorem \ref{thm_uni} and leave some open problems to the interested reader.
In Sections \ref{sec_gen} and \ref{sec-final} we present a more general
approach which provides slightly more general results than certain parts of
Theorem \ref{thm_uni}.

\bigskip

\section{The proof of Theorem \protect\ref{thm_uni}}

\label{sec_proof}

Let us recall a generalization of the Kahane-Salem-Zygmund inequality (see
\cite{alb, maton}) that will be crucial to prove the optimality of the
exponents:

\begin{KSZ}
Let $m,n\geq 1,\,(p_1, \ldots , p_m)\in \left[ 1,+ \infty \right]^m $ and
let us define
\begin{equation*}
\alpha (p):=
\begin{cases}
\displaystyle\frac{1}{2}-\frac{1}{p}\,,\text{if }p\geq 2; \\
0\,,\text{otherwise}.%
\end{cases}%
\end{equation*}
Then there exist an universal constant $C_{m}$ (depending only on $m$) and a
$m$-linear map $A:\ell _{p_1}^{n}\times \cdots \times \ell
_{p_m}^{n}\rightarrow \mathbb{K}$ of the form
\begin{equation*}
A\left( z^{(1)},\ldots ,z^{(m)}\right) =\sum_{i_{1},\ldots ,i_{m}=1}^{n}\pm
z_{i_{1}}^{(1)}\cdots z_{i_{m}}^{(m)},
\end{equation*}
such that
\begin{equation*}
\Vert A\Vert \leq C_{m}\cdot n^{\frac{1}{2}+\alpha (p_1)+ \cdots + \alpha
(p_m)}.
\end{equation*}
\end{KSZ}

\bigskip

From now on, $T:\ell _{p_1}^{n}\times \dots \times \ell_{p_m}^{n}\rightarrow
\mathbb{K}$ is an $m$-linear form.

\bigskip

\subsection{Case $\mathbf{p} \in [2,2m]^m$}

Let us suppose that $M^2_<$ is non-empty. For the sake of clarity we shall
assume that $M^2_< := \{ 1,\dots, k\}$. From Remark \ref{APps} for $r=2$, we
have
\begin{equation*}
\left( \sum_{i_1,\dots, i_m =1}^{n} |T(e_{i_1},\dots, e_{i_m})|^2 \right)^{%
\frac{1}{2}} \leq C n^{\left| \frac{1}{\mathbf{p}} \right|-\frac{1}{2}}.
\end{equation*}
Let $x_1, \ldots , x_k$ be such that
\begin{equation*}
\frac{1}{r_i} = \frac{1}{2} + \frac{1}{x_i} \qquad \text{for all } i=1,
\ldots , k.
\end{equation*}
Using H\"{o}lder's inequality for mixed sums (Theorem \ref{GenHol}) and the
classical inclusion for $\ell_p$ spaces we have
\begin{align*}
& \left( \sum_{i_{1}=1}^{n}\left( \dots \left( \sum_{i_{m}=1}^{n}\left\vert
T(e_{i_{1}},\dots ,e_{i_{m}})\right\vert ^{r_{m}}\right) ^{\frac{r_{m-1}}{%
r_{m}}}\dots \right) ^{\frac{r_{1}}{r_{2}}}\right) ^{\frac{1}{r_{1}}} \\
& \leq \left( \sum_{i_{1},\ldots ,i_{k}=1}^{n}\left( \sum_{i_{k+1}}^{n}\dots
\left( \sum_{i_{m}=1}^{n}\left\vert T(e_{i_{^{1}}},\ldots
,e_{i_{m}})\right\vert ^{r_{m}}\right) ^{\frac{r_{m-1}}{r_{m}}}\dots \right)
^{\frac{1}{r_{k+1}}2}\right) ^{\frac{1}{2}}\cdot \left(
\sum_{i_{1}=1}^{n}\left( \dots \left( \sum_{i_{k}=1}^{n}1^{x_{k}}\right) ^{%
\frac{x_{k-1}}{x_{k}}}\dots \right) ^{\frac{x_{1}}{x_{2}}}\right) ^{\frac{1}{%
x_{1}}} \\
& \leq\left( \sum_{i_{1},\ldots ,i_{m}=1}^{n}\left\vert
T(e_{i_{^{1}}},\ldots ,e_{i_{m}})\right\vert ^{2}\right) ^{\frac{1}{2}}\cdot
n^{\frac{1}{x_{1}}+\dots +\frac{1}{x_{k}}} \\
& \leq D_{m,(2,\ldots,2),\mathbf{p}}^{\mathbb{K}}\Vert T\Vert \cdot
n^{\left| \frac{1}{\mathbf{p}} \right|-\frac{1}{2}}\cdot n^{\frac{1}{x_{1}}%
+\dots +\frac{1}{x_{k}}} \\
& =D_{m,(2,\ldots,2),\mathbf{p}}^{\mathbb{K}}\Vert T\Vert \cdot n^{\left|
\frac{1}{\mathbf{p}} \right|-\frac{1}{2}+\frac{1}{r_{1}}+\cdots +\frac{1}{%
r_{k}}-\frac{k}{2}} \\
& =D_{m,(2,\ldots,2),\mathbf{p}}^{\mathbb{K}}\Vert T\Vert \cdot
n^{\sum_{j\in M_{\,<}^{2}}\frac{1}{r_{j}}+\left| \frac{1}{\mathbf{p}}
\right|-\frac{1}{2}\left(|M_<^{2}|+1\right)}.
\end{align*}

\bigskip Now we prove the optimality. Let us consider the $k$-linear form $%
A_k: \ell_{p_1}^n \times \dots \times \ell_{p_k}^n \to \mathbb{K}$ from the
Kahane-Salem-Zygmund inequality given by
\begin{equation*}
A_k(z^{(1)},\dots, z^{(k)}) = \sum_{i_1,\dots, i_k =1}^{n} \pm z^{(1)}_{i_1}
\dots z^{(k)}_{i_k}
\end{equation*}
which fulfils
\begin{equation*}
\|A_k\| \leq C_k n^{\frac{k+1}{2}-\left( \frac{1}{p_1}+ \cdots + \frac{1}{p_k%
}\right)}.
\end{equation*}
Let us define the $m$-linear form $B_m: \ell_{p_1}^n \times \dots \times
\ell_{p_m}^n \to \mathbb{K}$ by
\begin{equation*}
B_m(z^{(1)}, \ldots , z^{(m)}) = A_k(z^{(1)}, \ldots , z^{(k)})z^{(k+1)}_1
\dots z^{(m)}_1.
\end{equation*}
Clearly $\|B_m\|=\|A_k\|$. Also notice that
\begin{align*}
& \left( \sum_{i_{1}=1}^{n}\left( \dots \left( \sum_{i_{m}=1}^{n}\left\vert
B_m(e_{i_{1}},\dots ,e_{i_{m}})\right\vert ^{r_{m}}\right) ^{\frac{r_{m-1}}{
r_{m}}}\dots \right) ^{\frac{r_{1}}{r_{2}}}\right) ^{\frac{1}{r_{1}}} \\
& \quad = \left( \sum_{i_{1}=1}^{n}\left( \dots \left(
\sum_{i_{k}=1}^{n}\left\vert A_k(e_{i_{1}},\dots ,e_{i_{k}})\right\vert
^{r_{k}}\right) ^{\frac{r_{k-1}}{ r_{k}}}\dots \right) ^{\frac{r_{1}}{r_{2}}%
}\right) ^{\frac{1}{r_{1}}} \\
& \quad =n^{ \frac{1}{r_{1}}+\dots +\frac{1}{r_{k}}}.
\end{align*}
Let us suppose the result holds for some exponent $t>0$. Then,
\begin{equation*}
\left( \sum_{i_{1}=1}^{n}\left( \dots \left( \sum_{i_{m}=1}^{n}\left\vert
B_m(e_{i_{1}},\dots ,e_{i_{m}})\right\vert ^{r_{m}}\right) ^{\frac{r_{m-1}}{
r_{m}}}\dots \right) ^{\frac{r_{1}}{r_{2}}}\right) ^{\frac{1}{r_{1}}}\leq C
\Vert B_m\Vert \cdot n^{t}.
\end{equation*}%
Since $p_j\geq 2$ for all $j = 1,\dots,k$ we have
\begin{equation*}
n^{\frac{1}{r_{1}}+\dots +\frac{1}{r_{k}}} \leq C \Vert B_m\Vert n^{t} \leq
C C_{k} n^{t+\frac{k+1}{2}-\left(\frac{1}{p_1}+ \cdots + \frac{1}{p_k}%
\right)}.
\end{equation*}
Thus, we obtain a lower bound for the exponent $t$ that fulfills the result:
\begin{equation*}
t \geq \frac{1}{r_{1}}+\dots +\frac{1}{r_{k}} + \left(\frac{1}{p_1}+ \cdots
+ \frac{1}{p_k}\right)-\frac{k+1}{2} = \sum_{j\in M_{\,<}^{2}}\frac{1}{r_{j}}
+ \sum_{j\in M_{\,<}^{2}}\frac{1}{p_{j}} - \frac{|M_<^{2}|+1}{2}.
\end{equation*}
When $M_{<}^{2}= \{1,\dots,m\}$, this lower bound coincides with the
exponent we obtained before:
\begin{equation*}
s=\sum_{j=1}^{m} \frac{1}{r_j}+ \left| \frac{1}{\mathbf{p}} \right| - \frac{%
m+1}{2}
\end{equation*}
and, therefore, we gain the optimality for the exponent in this situation.

\bigskip

\bigskip

\subsection{Case $\left|\frac{1}{\mathbf{p}} \right| \leq \frac12$}

Similarly to the previous cases, we may suppose that $M^\text{HL}_< := \{
1,\dots, k\}$ is non-empty. Let $x_{1},\dots,x_{k} > 0$ be such that
\begin{equation*}
\frac{1}{r_{i}} =\frac{m+1-2 \left|\frac{1}{\mathbf{p}} \right|}{2m}+\frac{1%
}{x_{i}}, \ \ \ \text{ for all } \ i=1,\dots,k.
\end{equation*}

Write $\rho_{\text{{\tiny HL}}}:=2m/(m+1-2| \frac{1}{\mathbf{p}} | )$. Using
H\"{o}lder's inequality for mixed $\ell _{p}$ spaces (Theorem \ref{GenHol}),
the canonical inclusion of $\ell _{p}$ spaces and the classical
Hardy--Littlewood's inequality, we have
\begin{align*}
& \left( \sum_{i_{1}=1}^{n}\left( \dots \left( \sum_{i_{m}=1}^{n}\left\vert
T(e_{i_{1}},\dots ,e_{i_{m}})\right\vert ^{r_{m}}\right) ^{\frac{r_{m-1}}{
r_{m}}}\dots \right) ^{\frac{r_{1}}{r_{2}}}\right) ^{\frac{1}{r_{1}}} \\
&\leq \left( \sum_{i_{1},\ldots ,i_{k}=1}^{n} \left( \sum_{i_{k+1}=1}^{n}
\left( \dots \left( \sum_{i_{m}=1}^{n} \left\vert T(e_{i_{1}},\dots
,e_{i_{m}})\right\vert ^{r_{m}} \right)^{\frac{r_{m-1}}{r_{m}}} \dots
\right)^{\frac{r_{k+1}}{r_{k+2}}} \right) ^{\frac{\rho}{r_{k+1}}} \right)^%
\frac{1}{\rho} \times \\
& \times \left( \sum_{i_{1}=1}^{n}\left( \dots \left(
\sum_{i_{k}=1}^{n}1^{x_{k}}\right) ^{\frac{x_{k-1}}{x_{k}}}\dots \right) ^{
\frac{x_{1}}{x_{2}}}\right) ^{\frac{1}{x_{1}}} \\
& \leq \left( \sum_{i_{1},\ldots ,i_{m}=1}^{n}\left\vert
T(e_{i_{^{1}}},\ldots ,e_{i_{m}})\right\vert^{\frac{2m} {m+1-2 \left| \frac{1%
} {\mathbf{p}} \right|}} \right)^{\frac{m+1-2\left|\frac{1}{\mathbf{p}}%
\right|}{2m}}\cdot n^{\frac{1}{x_{1}}+\dots +\frac{1}{x_{k}}} \\
& \leq D_{m,(\rho_{\text{{\tiny HL}}}, \dots, \rho_{\text{{\tiny HL}}}),%
\mathbf{p}}^{\mathbb{K}} \Vert T\Vert \cdot n^{\frac{1}{r_{1}}+\dots + \frac{%
1}{r_{k}} - \frac{m+1-2\left|\frac{1}{\mathbf{p}}\right|}{2m}k} \\
& \leq D_{m,(\rho_{\text{{\tiny HL}}}, \dots, \rho_{\text{{\tiny HL}}}),%
\mathbf{p}}^{\mathbb{K}} \Vert T\Vert \cdot n^{\sum_{j\in M_{<}^{\text{HL}}}%
\frac{1}{r_j} - \frac{m+1-2\left|\frac{1}{\mathbf{p}}\right|}{2m} \cdot
|M_<^{\text{HL}}|}.
\end{align*}

\bigskip

The optimality of the case $M_{<}^{\text{HL}} = \{1,\dots,m\}$, which has
exponent
\begin{equation*}
s = \sum_{j=1}^{m} \frac{1}{r_j} +\left|\frac{1}{\mathbf{p}}\right| - \frac{%
m+1}{2},
\end{equation*}
follows by the same argument of the previous item (a): a standard use of the
$m$-linear form from the Kahane--Salem--Zygmund inequality. When $M_{<}^{%
\text{HL}} = \emptyset$, we have $s=0$. Thus the optimality of the result is
obvious since it is immediate that the inequality (\ref{thm_uni_ineq}) does
not hold if $n^s$ is replaced by $n^{t}$ with $t<0$.

\begin{remark}
Maybe the lack of optimality in the above results is a lack of (to the best
to the author's knowledge) a Kahane--Salem--Zygmund type inequality for this
context. The optimality or not of the estimates of the previous propositions
are, in our opinion, interesting open problems. In the next section, we will
give a different approach to the case $\left|\frac{1}{\mathbf{p}}\right|\leq
\frac 12$ in Theorem \ref{thm_uni}(2), and get an inequality similar to (\ref%
{thm_uni_ineq}) with optimal exponents $s$.
\end{remark}

\section{Getting optimality}

\label{sec_gen}

We have seen in the previous section that in the case $\left|\frac{1}{%
\mathbf{p}}\right|\leq \frac 12$ we get optimality of the exponent $s$ from (%
\ref{thm_uni_ineq}) in the extreme cases $M_{<}^{\text{HL}} = \{1,\dots,m\}$
or $M_{<}^{\text{HL}} = \emptyset$. Let us see that we can also get
optimality in the intermediates cases $0<|M_{<}^{\text{HL}}|<m$ for a
similar inequality (given in Theorem \ref{alt}). The price we have to pay in
order to gain optimality in these intermediate cases is double. First, the
result is stated just for $r_1,\ldots,r_m$ in the interval $[1,2]$ and not
for all $r_1,\ldots,r_m>0$. Second, the constant appearing in (\ref{opt})
depends a priori of the exponents $r_1,\ldots,r_m$. Recall that the constant
$D_{m,\mathbf{p}}^{\mathbb{K}}$ that comes from (\ref{thm_uni_ineq}) does
not depend on the $r_i$'s.

We begin with an elementary lemma:

\begin{lemma}
\label{8800}Let $r_{1},\ldots ,r_{m}\in \left( 0,2\right] $, $m\geq 2$ be a
positive integer and $\left|\frac{1}{\mathbf{p}}\right|\leq \frac{1}{2}$. If%
\begin{equation*}
\frac{1}{r_{1}}+\cdots +\frac{1}{r_{m}}>\frac{m+1}{2}-\left|\frac{1}{\mathbf{%
p}}\right|,
\end{equation*}%
then there are $s_{1},\ldots ,s_{m}\in \left[\left( 1-\left| \frac{1}{%
\mathbf{p}} \right| \right)^{-1},2\right]$ such that $s_{j}\geq r_{j}$ for
all $j=1,\ldots ,m$ and%
\begin{equation*}
\frac{1}{s_{1}}+\cdots +\frac{1}{s_{m}}=\frac{m+1}{2}-\left|\frac{1}{\mathbf{%
p}}\right|.
\end{equation*}
\end{lemma}

\begin{proof}
Note that since $\left|\frac{1}{\mathbf{p}}\right| \leq \frac{1}{2}$ we have
\begin{equation*}
1<\frac{1}{1-\left|\frac{1}{\mathbf{p}}\right|}\leq 2.
\end{equation*}%
We divide the proof in two cases: \newline
\textbf{First case. }Suppose that $r_{j_{0}}\leq \left( 1-\left| \frac{1}{%
\mathbf{p}} \right| \right)^{-1}$ for some $j_{0}.$ In this case we define%
\begin{equation*}
s_{j}=2\text{ for all }j\neq j_{0}
\end{equation*}%
and
\begin{equation*}
s_{j_{0}}=\frac{1}{1-\left|\frac{1}{\mathbf{p}}\right|}.
\end{equation*}

\noindent \textbf{Second case. }Suppose that $r_{j}>\left( 1-\left\vert
\frac{1}{\mathbf{p}}\right\vert \right) ^{-1}$ for all $j=1,\dots ,m$. Note
first that not all the $r_{j}$'s are $2$. Otherwise, we have $\frac{m}{2}>%
\frac{m+1}{2}-\left\vert \frac{1}{\mathbf{p}}\right\vert $, which
contradicts $\left\vert \frac{1}{\mathbf{p}}\right\vert \leq \frac{1}{2}$.
Set $N:=\{j\in \{1,\ldots ,m\}:r_{j}<2\}$. So, $r_{j}=2$ for all $j\notin N$%
. If we replace every $r_{j},j\in N$, by $2$ then
\begin{equation*}
\frac{m}{2}=\sum_{k\notin N}\frac{1}{r_{k}}+\sum_{k\in N}\frac{1}{2}\leq
\frac{m+1}{2}-\left\vert \frac{1}{\mathbf{p}}\right\vert .
\end{equation*}%
Let $j_{0}$ be the minimum $j\in N$ such that
\begin{equation*}
\sum_{k\notin N}\frac{1}{r_{k}}+\sum_{\overset{k\in N}{k<j_{0}}}\frac{1}{2}+%
\frac{1}{2}+\sum_{\overset{k\in N}{k>j_{0}}}\frac{1}{r_{k}}\leq \frac{m+1}{2}%
-\left\vert \frac{1}{\mathbf{p}}\right\vert .
\end{equation*}%
Then,
\begin{equation*}
\sum_{k\notin N}\frac{1}{r_{k}}+\sum_{\overset{k\in N}{k<j_{0}}}\frac{1}{2}+%
\frac{1}{r_{j_{0}}}+\sum_{\overset{k\in N}{k>j_{0}}}\frac{1}{r_{k}}>\frac{m+1%
}{2}-\left\vert \frac{1}{\mathbf{p}}\right\vert .
\end{equation*}%
By an Intermediate Value argument, there exists $\delta _{j_{0}}>0$ such
that $s_{j_{0}}:=r_{j_{0}}+\delta _{j_{0}}\in \left[ \left( 1-\left\vert
\frac{1}{\mathbf{p}}\right\vert \right) ^{-1},2\right] $ and
\begin{equation*}
\sum_{k\notin N}\frac{1}{r_{k}}+\sum_{\overset{k\in N}{k<j_{0}}}\frac{1}{2}+%
\frac{1}{r_{j_{0}}+\delta _{j_{0}}}+\sum_{\overset{k\in N}{k>j_{0}}}\frac{1}{%
r_{k}}=\frac{m+1}{2}-\left\vert \frac{1}{\mathbf{p}}\right\vert .
\end{equation*}%
Hence, we take the $s_{j}$'s as follows:
\begin{equation*}
s_{j}:=%
\begin{cases}
r_{j_{0}}+\delta _{j_{0}}, & \text{if }j=j_{0}; \\
2, & \text{if }j\in N\text{ is such that }j<j_{0}; \\
r_{j}, & \text{if }j\notin N \text{ or }j\in N\text{ is such that }j>j_{0}.%
\end{cases}%
\end{equation*}
\end{proof}

\bigskip

As we have anticipated, the following theorem complements part (2) of
Theorem \ref{thm_uni}:

\bigskip

\begin{theorem}
\label{alt} Let $m\geq 2$ be an integer. If $\left|\frac{1}{\mathbf{p}}%
\right|\leq \frac{1}{2}$ and $\mathbf{r} \in [1,2]^m$, then there is a
constant $D_{m,\mathbf{r},\mathbf{p}}^{\mathbb{K}}\geq 1$ (not depending on $%
n$) such that
\begin{equation}  \label{opt}
\left( \sum_{i_{1}=1}^{n}\left( \dots \left( \sum_{i_{m}=1}^{n}\left\vert
T(e_{i_{1}},\dots ,e_{i_{m}})\right\vert ^{r_{m}}\right) ^{\frac{r_{m-1}}{%
r_{m}}}\dots \right) ^{\frac{r_{1}}{r_{2}}}\right) ^{\frac{1}{r_{1}}} \leq
D_{m, \mathbf{r}, \mathbf{p}}^{\mathbb{K}} \cdot n^{\max \left\{ \left|
\frac{1}{\mathbf{r}} \right| -\frac{m+1}{2} + \left| \frac{1}{\mathbf{p}}
\right|, 0\right\}} \Vert T\Vert,
\end{equation}%
for all $m$-linear forms $T:\ell _{p_1}^{n}\times \cdots \times
\ell_{p_m}^{n} \to \mathbb{K}$ and all positive integers $n$. Moreover, the
exponent $\max \left\{ \left| \frac{1}{\mathbf{r}} \right| -\frac{m+1}{2}+
\left| \frac{1}{\mathbf{p}} \right|, 0\right\} $ is optimal.
\end{theorem}

\begin{proof}
The case
\begin{equation*}
\frac{1}{r_{1}}+\cdots +\frac{1}{r_{m}}\leq \frac{m+1}{2}-\left|\frac{1}{%
\mathbf{p}}\right|,
\end{equation*}
is precisely Theorem \ref{main111}. Let us suppose that
\begin{equation*}
\frac{1}{r_{1}}+\cdots +\frac{1}{r_{m}}>\frac{m+1}{2}-\left|\frac{1}{\mathbf{%
p}}\right|,
\end{equation*}%
and let $s_{1},\ldots,s_{m}$ be as in Lemma \ref{8800}. Let $%
x_{1},\ldots,x_{m}$ be such that
\begin{equation*}
\frac{1}{r_{i}}=\frac{1}{s_{i}}+\frac{1}{x_{i}},\ \ \ \text{ for all }%
i=1,\dots ,m.
\end{equation*}%
Using again H\"{o}lder's inequality for mixed $\ell _{p}$ spaces (Theorem %
\ref{GenHol}) and the generalized Hardy--Littlewood inequality (Theorem \ref%
{main111}) we have
\begin{align*}
& \left( \sum_{i_{1}=1}^{n}\left( \dots \left( \sum_{i_{m}=1}^{n}\left\vert
T(e_{i_{1}},\dots ,e_{i_{m}})\right\vert ^{r_{m}}\right) ^{\frac{r_{m-1}}{%
r_{m}}}\dots \right) ^{\frac{r_{1}}{r_{2}}}\right) ^{\frac{1}{r_{1}}} \\
& \leq \left( \sum_{i_{1}=1}^{n}\left( \dots \left(
\sum_{i_{m}=1}^{n}\left\vert T(e_{i_{1}},\dots ,e_{i_{m}})\right\vert
^{s_{m}}\right) ^{\frac{s_{m-1}}{s_{m}}}\dots \right) ^{\frac{s_{1}}{s_{2}}%
}\right) ^{\frac{1}{s_{1}}}\cdot \left( \sum_{i_{1}=1}^{n}\left( \dots
\left( \sum_{i_{m}=1}^{n}1^{x_{m}}\right) ^{\frac{x_{m-1}}{x_{m}}}\dots
\right) ^{\frac{x_{1}}{x_{2}}}\right) ^{\frac{1}{x_{1}}} \\
& \leq D_{m,s,p}^{\mathbb{K}}\Vert T\Vert \cdot n^{\frac{1}{x_{1}}+\dots +%
\frac{1}{x_{m}}} \\
& =D_{m,s,p}^{\mathbb{K}}\Vert T\Vert \cdot n^{\frac{1}{r_{1}}+\dots +\frac{1%
}{r_{m}} - \left( \frac{1}{s_{1}}+\dots +\frac{1}{s_{m}} \right)} \\
& =D_{m,s,p}^{\mathbb{K}}\Vert T\Vert \cdot n^{\max \left\{ \left|\frac{1}{%
\mathbf{p}}\right| -\frac{m+1}{2}+\frac{1}{r_{1}}+\cdots +\frac{1}{r_{m}}%
,0\right\} }.
\end{align*}%
The optimality is proved as in the previous sections, using the
Kahane--Salem--Zygmund inequality.
\end{proof}

\begin{remark}
Whenever $\left|\frac{1}{\mathbf{p}}\right|\leq \frac 12$ and $1\leq
r_1,\ldots,r_m\leq 2m/(m+1-2|1/\mathbf{p}|)$ then we can apply both, Theorem %
\ref{thm_uni} (2) and Theorem \ref{alt}. The exponents we get in both cases
coincide since in this case $M_<^{HL}=\{1,\ldots,m\}$.

But there is still room to apply both results whenever $\left|\frac{1}{%
\mathbf{p}}\right|\leq \frac 12$ and some of the $r_j$ are between $%
2m/(m+1-2|1/\mathbf{p}|)$ and $2$. On one hand, if all $r_j$ are greater
than or equal to $2m/(m+1-2|1/\mathbf{p}|)$ then $M_<^{HL}$ is empty and
both exponent $s$ are equal to $0$ (as expected). On the other hand, if some
of the $r_j$ are bigger and some other are smaller than $2m/(m+1-2|1/\mathbf{%
p}|)$, we get different exponents, being the $s=s_{\ref{alt}}$ from Theorem %
\ref{alt} smaller than the exponent $s=s_{\ref{thm_uni}}$ from Theorem \ref%
{thm_uni}, and so, better. Indeed, if $r_1,\ldots,r_m\in [1,2]$ are such
that $0<k:=|M_<^{HL}|<m$ then,
\begin{equation*}
\sum_{j\in M_\geq^{HL}}\frac 1{r_j}\leq \frac{m+1-2\left|\frac{1}{\mathbf{p}}%
\right|}{2m}(m-k)=\left(\frac{m+1}2-\left|\frac{1}{\mathbf{p}}%
\right|\right)\left(1-\frac km\right),
\end{equation*}
and this inequality is equivalent to being $s_{\ref{alt}}\leq s_{\ref%
{thm_uni}}$.
\end{remark}

\bigskip

\section{Approach for $p_m \in (1,2]$}

\label{sec-final}

The paper \cite{ANPS_new} fully describes several inequalities involving
operators on $\ell_p$ spaces, some of which are useful for our purpose in
this paper. For all positive integers $m,\,k=1,\dots,m$, let us define
\begin{equation*}
\delta_{m-k+1}^{p_k, \dots, p_m} := \frac{1}{1-\left( \frac{1}{p_k}+ \cdots
+ \frac{1}{p_m} \right)}.
\end{equation*}
The following is one of the main results from \cite{ANPS_new}:

\begin{theorem}
\label{delta} Let $m\geq2$ be a positive integer, $q_1, \dots, q_m >0$ and $%
1<p_m \leq 2 < p_1, \dots, p_{m-1}$ be such that $\left| \frac{1}{\mathbf{p}}
\right| < 1$. The following assertions are equivalent:

\noindent (a) There exists $C_{p_1, \dots, p_m}\geq 1$ such that
\begin{equation*}
\left( \sum_{i_{1}=1}^{+\infty} \left( \ldots \left(
\sum_{i_{m}=1}^{+\infty}\left\vert A\left( e_{i_{1}},\ldots
,e_{i_{m}}\right) \right\vert ^{q_{m}} \right) ^{\frac{q_{m-1}}{q_{m}}%
}\ldots \right) ^{\frac{q_{1}}{q_{2}}} \right) ^{\frac{1}{q_{1}}} \leq
C_{p_1, \dots, p_m} \|A\|
\end{equation*}
for all bounded $m$-linear operator $A:\ell_{p_1} \times \cdots \times
\ell_{p_m} \to \mathbb{K}$.

\bigskip

\noindent(b) The exponents $q_1, \dots, q_m >0$ satisfy
\begin{equation*}
q_1 \geq \delta_{m}^{p_1,\dots,p_m},
q_2\geq\delta_{m-1}^{p_2,\dots,p_m},\dots, q_{m-1}\geq
\delta_{2}^{p_{m-1},p_m}, q_m\geq \delta_{1}^{p_m}.
\end{equation*}
\end{theorem}

The following Lemma is a variant from \cite[Lemma 2.1]{ANPS_new}.

\begin{lemma}
\label{lema2} Let $p_1, \dots, p_m \in [1,+\infty]$ and $q_1, \dots, q_m,
t_1, \dots, t_m, s_2, \dots, s_m \in (0,+\infty)$. Let us consider the
following properties:

\bigskip

\noindent (a) If there is $D_{p_2, \dots, p_m}>0$ such that, for all $(m-1)$%
-linear forms $A:\ell_{p_2} \times \dots \times \ell_{p_m} \to \mathbb{K}$,
\begin{equation*}
\left( \sum_{i_{2}=1}^{n_2} \left(\sum_{i_3=1}^{n_3} \ldots \left(
\sum_{i_{m}=1}^{n_m}\left\vert A\left( e_{i_{2}},\ldots ,e_{i_{m}}\right)
\right\vert ^{q_{m}} \right) ^{\frac{q_{m-1}}{q_{m}}}\ldots \right) ^{\frac{%
q_{2}}{q_{3}}} \right) ^{\frac{1}{q_{2}}} \leq D_{p_2, \dots, p_m} n_2^{t_2}
\cdot\dots \cdot n_m^{t_m} \|A\|
\end{equation*}
holds, then $t_i \geq s_i$ whenever some property $P_i$ is satisfied, for
all $i\in \{ 2, \dots , m\}$.

\bigskip

\noindent (b) If there is $C_{p_1, \dots, p_m}>0$ such that, for all $m$%
-linear forms $B: \ell_{p_1} \times \dots \times \ell_{p_m} \to \mathbb{K}$,
\begin{equation*}
\left( \sum_{i_{1}=1}^{n_1} \left(\sum_{i_2=1}^{n_2} \ldots \left(
\sum_{i_{m}=1}^{n_m}\left\vert B\left( e_{i_{1}},\ldots ,e_{i_{m}}\right)
\right\vert ^{q_{m}} \right) ^{\frac{q_{m-1}}{q_{m}}}\ldots \right) ^{\frac{%
q_{1}}{q_{2}}} \right) ^{\frac{1}{q_{1}}} \leq C_{p_1, \dots, p_m} n_1^{t_1}
\cdot\dots \cdot n_m^{t_m} \|B\|
\end{equation*}
holds, then $t_i \geq s_i$ whenever the property $P_i$ is satisfied, for all
$i\in \{ 2, \dots , m\}$.

\bigskip

\noindent Then property (a) implies property (b).
\end{lemma}

We have an optimal result when $1<p_m \leq 2$ and $p_1,\dots,p_{m-1}> 2$:

\begin{theorem}
Let $m\geq 2,\, \mathbf{r}:=(r_1, \dots, r_m) \in (0, +\infty)^m$ and $1<
p_m \leq 2 < p_1, \dots, p_{m-1}$ with $\frac{1}{p_1} + \cdots + \frac{1}{p_m%
} < 1$. Then, there exists $D^{\mathbb{K}}_{m, \mathbf{p}, \mathbf{r}} \geq 1
$ such that, for all positive integers $n_1,\dots,n_m$ and for all bounded $m
$-linear form $T:\ell_{p_1} \times \cdots \times \ell_{p_m} \to \mathbb{K}$,
\begin{equation*}
\left( \sum_{i_{1}=1}^{n} \left( \ldots \left( \sum_{i_{m}=1}^{n} \left\vert
T\left( e_{i_{1}},\dots ,e_{i_{m}}\right) \right\vert ^{r_{m}} \right)^{%
\frac{r_{m-1}}{r_{m}}}\ldots \right)^{\frac{r_{1}}{r_{2}}} \right)^{\frac{1}{%
r_{1}}} \leq D^{\mathbb{K}}_{m, \mathbf{p},\mathbf{r}} \cdot \|T\| \cdot
\prod_{k=1}^{m} n_k^{ \max \left\{ \frac{1}{ r_k}-\frac{1}{%
\delta_{m-k+1}^{p_k, \dots, p_m} },0 \right\} }.
\end{equation*}
Moreover, the exponent $\max \left\{ \frac{1}{r_m} - \frac{1}{%
\delta_{1}^{p_m} },0 \right\}$ is optimal and, for each $k=1,\dots,m-1$, the
exponent $\max \left\{ \frac{1}{ r_k}-\frac{1}{\delta_{m-k+1}^{p_k, \dots,
p_m} },0 \right\}$ is optimal if $r_j \geq \delta_{m-j+1}^{p_j, \dots, p_m}$
for $k+1 \leq j \leq m$.
\end{theorem}

\begin{proof}
Firstly, notice that when $r_k \geq \delta_{m-k+1}^{p_k, \dots, p_m}$ for
all $k = 1,\dots,m$ we have precisely Theorem \ref{delta}. Therefore, we may
suppose that there is some exponent $r_j < \delta_{m-j+1}^{p_j, \dots, p_m}$
and, for the sake of clarity, we may consider that this happens for all
exponents. Let $x_1, \dots, x_m>0$ be such that
\begin{equation*}
\frac{1}{r_k} = \frac{1}{\delta_{m-k+1}^{p_k, \dots, p_m}} + \frac{1}{x_k},
\quad \text{ for all } k=1, \dots, m.
\end{equation*}
By H\"older's inequality and Theorem \ref{delta}, we have
\begin{align*}
& \left( \sum_{i_{1}=1}^{n_1}\left( \dots \left(
\sum_{i_{m}=1}^{n_m}\left\vert T(e_{i_{1}},\dots ,e_{i_{m}})\right\vert
^{r_{m}}\right) ^{\frac{r_{m-1}}{r_{m}}}\dots \right) ^{\frac{r_{1}}{r_{2}}%
}\right) ^{\frac{1}{r_{1}}} \\
&\leq \left( \sum_{i_{1}=1}^{n_1}\left( \dots \left(
\sum_{i_{m}=1}^{n_m}\left\vert T(e_{i_{1}},\dots ,e_{i_{m}})\right\vert
^{\delta_{1}^{p_m}}\right) ^{\frac{\delta_{2}^{p_{m-1}, p_m}}{%
\delta_{1}^{p_m}}}\dots \right) ^{\frac{\delta_{m}^{p_1, \dots, p_m}}{%
\delta_{m-1}^{p_2, \dots, p_m}}}\right) ^{\frac{1}{\delta_{m}^{p_1, \dots,
p_m}}} \times \\
& \times \left( \sum_{i_{1}=1}^{n_1}\left( \dots \left(
\sum_{i_{m}=1}^{n_m}1^{x_{m}}\right) ^{\frac{x_{m-1}}{x_{m}}}\dots \right)
^{ \frac{x_{1}}{x_{2}}}\right) ^{\frac{1}{x_{1}}} \\
& \leq C_{p_1, \dots, p_m} \|T\| \cdot n_1^{\frac{1}{x_{1}}} \cdot n_2^{%
\frac{1}{x_2}} \cdot \dots \cdot n_m^{\frac{1}{x_{m}}} \\
& \leq C_{p_1, \dots, p_m}\Vert T\Vert \cdot n_1^{\frac{1}{r_1} - \frac{1}{%
\delta_{m}^{p_1, \dots, p_m}}} \cdot n_2^{\frac{1}{r_2} -\frac{1}{%
\delta_{m-1}^{p_2, \dots, p_m}}} \cdot \dots \cdot n_m^{\frac{1}{r_m} -
\frac{1}{\delta_{1}^{p_m}}}.
\end{align*}

To obtain the optimality, we will proceed by induction on $m$. Initially we
prove the case $m=2$. Let us suppose that there exist $t_1,t_2>0$ that
fulfils
\begin{equation*}
\left( \sum_{i=1}^{n_{1}}  \left( \sum_{j=1}^{n_{2}}  \left|U\left(e_{i},
e_{j}\right)\right|^{r_{2}}  \right)^{\frac{r_1}{r_2}} \right)^{\frac{1}{r_1}%
} \leq D_{2,\mathbf{p}, \mathbf{r}}^{\mathbb{K}} n_{1}^{t_{1}} n_{2}^{t_{2}}
\|U\|,
\end{equation*}
for all bounded bilinear forms $U: \ell_{p_1} \times \ell_{p_2} \to \mathbb{K%
}$, with $1< p_2 \leq 2 < p_1$ and $\frac{1}{p_1} + \frac{1}{p_2} < 1$. By
considering the bilinear form $T:\ell _{p_{1}} \times \ell _{p_{2}} \to
\mathbb{K}$ given by $T(x,y) = x_1\sum_{j=1}^{n_{2}}y_{j}$, we have $\|T\|
\leq n_{2}^{1-\frac{1}{p_{2}}} = n_{2}^{1/\delta_{1}^{p_2}}$. Then,
\begin{equation*}
n_{2}^{\frac{1}{r_{2}}} \leq D_{2,\mathbf{p}, \mathbf{r}}^{\mathbb{K}} \cdot
n_{1}^{t_1} n_2^{t_2} \|T\| \leq D_{2,\mathbf{p}, \mathbf{r}}^{\mathbb{K}}
\cdot n_{1}^{t_1} n_2^{t_2} n_{2}^{1/\delta_{1}^{p_2}}.
\end{equation*}
Thus, making $n_{2}\rightarrow \infty $ we gain
\begin{equation*}
t_{2} \geq \max \left\{ \frac{1}{r_2} - \frac{1}{\delta_{1}^{p_2} },0
\right\}.
\end{equation*}
Now let us prove the optimality of the exponent $\max \left\{ \frac{1}{r_1}
- \frac{1}{\delta_{2}^{p_1, p_2}}, 0 \right\}$ of $n_1$ when $r_2 \geq
\delta_{1}^{p_2}$. The bilinear form $D:\ell _{p_{1}} \times \ell _{p_{2}}
\to \mathbb{K}$ given by $D(x,y) = \sum_{j=1}^{n} x_{j} y_{j}$, satisfies $%
\|D\| \leq n^{1-\frac{1}{p_{1}} - \frac{1}{p_{2}}} =
n^{1/\delta_{2}^{p_1,p_2}}$. If there exists $t_1$ such that
\begin{equation*}
\left( \sum_{i=1}^{n_{1}} \left( \sum_{j=1}^{n_{2}} \left|U\left(e_{i},
e_{j}\right)\right|^{r_{2}} \right)^{\frac{r_1}{r_2}} \right)^{\frac{1}{r_1}%
} \leq D_{2,\mathbf{p}, \mathbf{r}}^{\mathbb{K}} n_{1}^{t_{1}} \|U\|,
\end{equation*}
holds for all $n_1,n_2$ and all bounded bilinear forms $U: \ell_{p_1} \times
\ell_{p_2} \to \mathbb{K}$, then considering $n_1=n_2=n$, we get
\begin{equation*}
n^{\frac{1}{r_1}} = \left( \sum_{i=1}^{n}  \left( \sum_{j=1}^{n}
\left|D\left(e_{i}, e_{j}\right)\right|^{r_{2}}  \right)^{\frac{r_1}{r_2}}
\right)^{\frac{1}{r_1}} \leq D_{2,\mathbf{p}, \mathbf{r}}^{\mathbb{K}}
n^{t_{1}} \|D\| \leq D_{2,\mathbf{p}, \mathbf{r}}^{\mathbb{K}} n^{t_{1}}
n^{1/\delta_{2}^{p_1,p_2}}.
\end{equation*}
Thus, making $n \rightarrow +\infty $ we conclude the argument for $m=2$:
\begin{equation*}
t_1 \geq \max \left\{\frac{1}{r_1} - \frac{1}{\delta_{2}^{p_1,p_2}}, 0
\right\}.
\end{equation*}

Now we suppose the result holds for all bounded $(m-1)$-linear forms. By
induction hypothesis, if
\begin{equation*}
\left( \sum_{i_{2}=1}^{n_{2}}\left( \sum_{i_{3}=1}^{n_{3}}\ldots \left(
\sum_{i_{m}=1}^{n_{m}}\left\vert A\left( e_{i_{2}},\ldots ,e_{i_{m}}\right)
\right\vert ^{r_{m}}\right) ^{\frac{r_{m-1}}{r_{m}}}\ldots \right) ^{\frac{%
r_{2}}{r_{3}}}\right) ^{\frac{1}{r_{2}}}\leq D_{p_{2},\dots
,p_{m}}n_{2}^{t_{2}}\cdot \dots \cdot n_{m}^{t_{m}}\Vert A\Vert
\end{equation*}
holds for all bounded $(m-1)$-forms $A:\ell _{p_{2}}\times \dots \times \ell
_{p_{m}} \to \mathbb{K}$, we have
\begin{equation*}
t_{m} \geq \max \left\{ \frac{1}{r_m} - \frac{1}{\delta_{1}^{p_m} },0
\right\}
\end{equation*}
and
\begin{equation}
\forall\, k=2,\dots,m-1,\ t_k \geq \max \left\{ \frac{1}{ r_k}-\frac{1}{%
\delta_{m-k+1}^{p_k, \dots, p_m} },0 \right\}, \ \text{whenever} \ r_j \geq
\delta_{m-j+1}^{p_j, \dots, p_m},\, \forall\, j=k+1,\dots,m.  \label{ind-hyp}
\end{equation}
Then we suppose that there exist $t_{1},\dots ,t_{m} \geq0$ such that
\begin{equation}
\left( \sum_{i_{1}=1}^{n_{1}}\left( \ldots \left(
\sum_{i_{m}=1}^{n_{m}}\left\vert T\left( e_{i_{1}},\ldots ,e_{i_{m}}\right)
\right\vert ^{r_{m}}\right) ^{\frac{r_{m-1}}{r_{m}}}\ldots \right) ^{\frac{%
r_{1}}{r_{2}}}\right) ^{\frac{1}{r_{1}}} \leq D_{m,\mathbf{p},\mathbf{r}}^{%
\mathbb{K}}\cdot n_{1}^{t_{1}}\cdot n_{2}^{t_{2}}\cdot \dots \cdot
n_{m}^{t_{m}}\Vert T\Vert .  \label{desotim}
\end{equation}%
for all bounded $m$-linear forms $T:\ell_{p_{1}} \times \dots \times
\ell_{p_{m}} \to \mathbb{K}$. By Lemma \ref{lema2} we gain \eqref{ind-hyp}.
It remains to prove the estimate for $t_{1}$ under the condition $r_j \geq
\delta_{m-j+1}^{p_j, \dots, p_m}$ for all $j=2,\dots,m$, which means that we
can take $t_2=\cdots=t_m=0$. Let us consider the $m$-linear form $%
D_{m}^{n}:\ell _{p_{1}}\times \dots \times \ell _{p_{m}}\to \mathbb{K}$
defined by
\begin{equation*}
D_{m}^{n}(x^{(1)},\dots ,x^{(m)}) =\sum_{j=1}^{n}x_{j}^{(1)}x_{j}^{(2)}
\cdots x_{j}^{(m)}.
\end{equation*}
Notice that $\|D_{m}^{n}\| \leq n^{1-\left(\frac{1}{p_{1}} + \cdots +\frac{1%
}{p_{m}} \right)} = n^{1/\delta_m^{p_1,\dots,p_m}}$. By considering $%
n=n_1=\dots=n_m$, from (\ref{desotim}) we obtain
\begin{align*}
n^{\frac{1}{r_{1}}} = \left( \sum_{i=1}^{n} \left|
D_{m}^{n}(e_i,\dots,e_i)\right|^{r_1} \right)^{\frac{1}{r_1}} & = \left(
\sum_{i_{1}=1}^{n_{1}} \left( \dots \left( \sum_{i_{m}=1}^{n_m} \left|
D_{m}^{n} \left( e_{i_{1}},\ldots ,e_{i_{m}}\right) \right|^{r_m} \right)^{%
\frac{r_{m-1}}{r_m}} \dots \right)^{\frac{r_{1}}{r_2}} \right)^{\frac{1}{%
r_{1}}} \\
& \leq D_{m,\mathbf{p},\mathbf{r}}^{\mathbb{K}} \cdot n^{t_{1}} \cdot n^{0}
\cdots n^{0} \|D_{m}^{n}\| \\
& \leq D_{m,\mathbf{p},\mathbf{r}}^{\mathbb{K}} \cdot n^{t_{1}} \cdot
n^{1/\delta_m^{p_1,\dots,p_m}}.
\end{align*}
Taking $n \rightarrow +\infty$, we conclude the proof
\begin{equation*}
t_1 \geq \max \left\{ \frac{1}{r_1}-\frac{1}{\delta_{m}^{p_1,\dots, p_m}},0
\right\}.
\end{equation*}
\end{proof}

\section{Final remark: the linear case}

A natural question rises after all the case we dealt: \emph{what is the
behaviour of the linear case ($m=1$)?} On this situation, the result is
clear and sharp: given a positive integer $n,\, p\in[1,+\infty]$ and $r \in
(0,+\infty)$,
\begin{equation}
\left(\sum_{i=1}^{n}\left|T(e_i)\right|^r\right)^{\frac{1}{r}} \leq n^{\max
\left\{\frac{1}{r}-\frac{1}{p^{\prime }},0\right\}} \|T\|,  \label{linear}
\end{equation}
holds for all linear forms $T:\ell_{p}^n\to \mathbb{K}$, where $p^{\prime }$
denote the conjugate exponent of $p$. Moreover, the exponent $\max \left\{%
\frac{1}{r}-\frac{1}{p^{\prime }},0\right\}$ is optimal. Indeed, when $r\geq
p^{\prime }$ \eqref{linear} is obvious. When $r < p^{\prime }$, %
\eqref{linear} follows by H\"older's inequality, and the exponent optimality
follows by the inclusions of $\ell_q$ spaces.


\begin{thebibliography}{99}
\bibitem{four} R.A. Adams and J.J.F. Fournier, Sobolev spaces, Elsevier,
Second Edition, 2003.

\bibitem{alb} N. Albuquerque, F. Bayart, D. Pellegrino and J. Seoane--Sep%
\'{u}lveda, Sharp generalizations of the multilinear Bohnenblust--Hille
inequality, J. Funct. Anal. \textbf{266} (2014), 3726--3740.

\bibitem{isr} N. Albuquerque, F. Bayart, D. Pellegrino and J. Seoane--Sep%
\'{u}lveda, Optimal Hardy--Littlewood inequalities for polynomials and
multilinear operators, Israel Journal of Mathematics, \textbf{211} (2016),
197-220.

\bibitem{archiv} G. Ara\'{u}jo, D. Pellegrino, Optimal Hardy--Littlewood
type inequalities for $m$-linear forms on $\ell _{p}$ spaces with $1\leq
p\leq m$, Archiv der Math. 105 (2015), \ 285--295.

\bibitem{ANPS_new} R. Aron, D. N\'u\~nez-Alarc\'on, D. Pellegrino and D.
Serrano, Optimal exponents for Hardy-Littlewood inequalities for $m$-linear
operators, arXiv:1602.00178v2 [math.FA] 11 May 2016.

\bibitem{bene} A. Benedek, R. Panzone, The space $L_{p}$, with mixed norm.
Duke Math. J. 28 1961 301--324.

\bibitem{bh} H.F. Bohnenblust and E. Hille, On the absolute convergence of
Dirichlet series, Ann. of Math. (2) \textbf{32} (1931), 600--622.

\bibitem{wa} W. Cavalcante, D. N\'{u}\~{n}ez-Alarc\'{o}n, Remarks on an
Inequality of Hardy and Littlewood, to appear in Quast. Math.

\bibitem{dim} V. Dimant and P. Sevilla--Peris, Summation of coefficients of
polynomials on  $\ell _{p}$ spaces. Publ. Mat. 60 (2016), no. 2, 289--310.

\bibitem{hl} G. Hardy and J. E. Littlewood, Bilinear forms bounded in space $%
[p,q]$, Quart. J. Math. \textbf{5} (1934), 241--254.

\bibitem{lux} W.A.J. Luxemburg, Banach Function Spaces, Essen, 1955.

\bibitem{maton} A. M. Mantero and A. Tonge, The Schur multiplication in
tensor algebras, Stud. Math. \textbf{68} (1980), no. 1, 1--24.

\bibitem{pp} T. Praciano-Pereira, On bounded multilinear forms on a class of
lp spaces. J. Math. Anal. Appl. 81 (1981), no. 2, 561--568.
\end{thebibliography}
\end{document}